\documentclass[reqno]{amsart}
\usepackage{amsmath}
\usepackage{amsfonts}
\usepackage{amsthm}
\usepackage{amssymb}
\usepackage{enumerate}
\usepackage[all]{xy}
\usepackage{cite}
\usepackage{mathrsfs}
\usepackage{color}
\usepackage{xcolor}
\usepackage{hyperref}
\usepackage{fancyhdr}
\hypersetup{
	colorlinks=true,
	anchorcolor=blue,
	linkcolor=blue,
	filecolor=blue,
	urlcolor=blue,    
	citecolor=blue,
	bookmarks=true,
	bookmarksopen=true,
	pdfborder=000
}
\numberwithin{equation}{section}
\theoremstyle{plain}
\newtheorem{thm}{Theorem}[section]

\newtheorem{proposition}[thm]{Proposition}

\newtheorem{cor}[thm]{Corollary}

\newtheorem{lemma}[thm]{Lemma}
\newtheorem{remark}{Remark}
\newtheoremstyle{noparens}%
{}{}%
{\itshape}{}%
{\bfseries}{.}%
{ }%
{\thmname{#1}\thmnumber{ #2}\mdseries\thmnote{ #3}}
\theoremstyle{noparens}

\theoremstyle{definition}

\theoremstyle{remark}

\makeatletter
\newcommand{\bfrho}{\boldsymbol{\rho}}

\newcommand{\Rmnum}[1]{\expandafter\@slowromancap\romannumeral #1@}
\@namedef{subjclassname@2020}{\textup{2020} Mathematics Subject Classification}
\makeatother
\makeatletter
\def\@eqnnum{{\normalfont \normalcolor \hfill \@eqnnumformat{(\theequation)}}}
\makeatother

\usepackage{etoolbox}
\patchcmd{\section}{\scshape}{}{}{}
\patchcmd{\subsection}{\bfseries\scshape}{\bfseries}{}{}
\patchcmd{\subsubsection}{\bfseries\scshape}{\bfseries}{}{}

\makeatletter
\def\@settitle{\begin{center}\normalfont\LARGE\bfseries \@title\end{center}}
\makeatother

\begin{document}
	
\title{BMO on Weighted Bergman Spaces over Tubular Domains}

	\author[Jiaqing Ding, Haichou Li, Zhiyuan Fu, Yanhui Zhang]{ Jiaqing Ding$^{1}$, Haichou Li$^{2,*}$, Zhiyuan Fu$^3$, YanHui Zhang$^4$}
	
	\address{$1.$ College of Mathematics and informatics, South China Agricultural University, Guangzhou, 510640, China}
	\address{$2.$ College of Mathematics and informatics, South China Agricultural University, Guangzhou, 510640, China}
		\address{$3.$ College of Mathematics and Statistics, Beijing Technology and Business University, Beijing, 100048, China}
	\address{$4.$ College of Mathematics and Statistics, Beijing Technology and Business University, Beijing, 100048, China}
	\email{djq123@stu.scau.edu.cn (J. Ding)\:; hcl2016@scau.edu.cn (H. Li)\: ; 435010980@qq.com(Z. Fu) \: ; zhangyanhui@th.btbu.edu.cn(Y. Zhang)}.

	\thanks{$^*$Corresponding author. The research of the second author was supported by NSFC (Grant No. 12326407 and 12071155); the fourth author was supported by NSFC (No.11971042).}
	
	\begin{abstract}
		
		In this paper, we characterize Bounded Mean Oscillation  (BMO)  and establish their  connection with Hankel operators on weighted Bergman spaces over tubular domains. By utilizing the space BMO, we provide a new characterization of Bloch spaces on tubular domains. Next, we define a modified projection operator and prove its boundedness. Furthermore, we introduce differential operators and demonstrate that these operators belong to Lebesgue spaces on tubular domains. Finally, we establish an integral representation for Bergman functions  using these  differential operators.
	\end{abstract}
	\maketitle
 	
\section{Introduction}

 BMO plays a critical role in harmonic analysis and has been extensively studied, with significant contributions from researchers such as Garnett \cite{JB} and Fefferman-Stein \cite{CE}. Zhu \cite{Zhu1992BMOAH} studied BMO in the unit ball, providing an equivalent characterization that has influenced later research. Békollé et al. \cite{BEKOLLE1990310} extended the investigation to BMO on bounded symmetric domains, while  Pau et al. \cite{Pau2016} considered the weighted BMO within the unit ball.

The Bloch space is characterized by its invariance under biholomorphic transformations and Möbius maps, highlighting its significance in complex analysis.
And its properties have been studied by many authors (\cite{JM1},\cite{JM2},\cite{SA},\cite{Chu}, \cite{Zhu05}, \cite{Tim80I}, \cite{Bek} and \cite{Si}), with extensions  from the unit disk to the unit ball, and further to the unbounded domain of the Siegel upper half-plane. For the cases of  the unit disk and unit ball, many foundational results have established equivalences with $L^\infty$ spaces. Further studies in the Siegel upper half-plane have explored analogous characterizations, although these come with limitations for specific range of $p$. In this paper, we mainly consider another unbounded domains, such as tubular domain $T_B$. Following the approach taken in the cases of the unit ball and Siegel upper half-plane, we utilize invariant gradients to define the Bloch space $\mathcal{B}$ over $T_B$.

The main ideas of this paper are inspired by the works of Liu\cite{Liu} and Zhu \cite{Zhu1992BMOAH},  with Zhu’s work in particular investigating key results concerning BMO and Bloch space within the unit ball. In contrast, our work extends these results to tubular domains, providing a broader framework for understanding the relationship between these spaces and their applications. While Si\cite{Si} has explored similar results to the Siegel upper half-plane. 

In this paper, we obtain four main results. First, we establish a characterization of the $BMO^p_\alpha$ on $T_B$ by the bounded Hankel operators on weighted Bergman spaces. Utilizing $BMO^p_\alpha$ , we derive a new characterization of Bloch space. Additionally, we demonstrate that the Bergman kernel does not belong to ${L}^{1}_\alpha(T_{B})$ within the tubular domains, thereby introducing a modified kernel function  $\widetilde{K}_\alpha$ and a modified projection operator $\widetilde{P}_\alpha$. We also prove that the modified projection from $L^{\infty }(T_{B})$ to $\widetilde{\mathcal{B}}$ is bounded and provide a kind of integral representation for Bergman functions.

{\bfseries Theorem A}\label{thm:mainA}
For $r>0$, $1\leq p<\infty $, and $p(\alpha+1)>\lambda+1$, let $f\in L^p_\lambda(T_B)$, the following conditions are equivalent:
\begin{enumerate}[(a)]
	\item $f\in BMO^p_r$.
	\item $f=f_1+f_2$ with $f_1\in BO$ and $f_2\in BA^p$.
	\item $H^{(\alpha)}_f$ and $H^{(\alpha)}_{\overline f}$ are both bounded on $A^p_\lambda(T_B)$.
	
\end{enumerate}

Following the characterization of $BMO^p_r$ in Theorem A, we now investigate the relationship between BMO and the Bloch spaces defined on tubular domains.

{\bfseries Theorem B}\label{thm:mainB}
	Let $H(T_{B})$ be the space of holomorphic functions in $T_{B}$. For any $r>0$, we have 
	\[
	\mathcal{B}=H(T_{B})\cap BMO^p_r ,
	\]
	and 
	\[
	\mathcal{B}_0=H(T_{B})\cap VMO^p_r.
	\]
	Moreover, $\|f\|_{\mathcal{B}}$ and $\|f\|_{BMO^p_r}$ are equivalent.
	
 Next, we will consider a new space 
$$\widetilde{\mathcal{B}}:=\left \{f\in \mathcal{B}:f(\mathbf{i})=0 \right \},$$
where $\mathbf{i}:=(0',i)$, and $0'=(0,\cdots,0)\in \mathbb{R}^{n-1}$.

The Bloch space on the unit ball is known to be equivalent to $L^{\infty}(\mathbb{B}_n)$ when considered under the projection operator $P_\alpha$. We aim to extend this relationship to tubular domains. However, due to Forreli-Rudin type estimates showing that the Bergman kernel does not belong to ${L}^{1}_\alpha(T_{B})$ on tubular domains, it becomes necessary to define a modified kernel function, $\widetilde{K_\alpha}$, and a corresponding modified projection operator,  $\widetilde{P_\alpha}$, from ${L}^{\infty }(T_B)$ to  $\widetilde{\mathcal{B}}$,  as described below:
$$\widetilde{K_\alpha}:=K_\alpha(z,w)-K_\alpha(\mathbf{i},w),$$
$$\widetilde{P_\alpha}f(z):=\int_{T_B}\widetilde{K_\alpha}(z,w)f(w)dV_\alpha(w), $$
where $\mathbf{i}:=(0',i)$.

{\bfseries Theorem C}\label{thm:mainC}
	$\widetilde{P}_\alpha$ is a bounded projection from $L^{\infty }(T_{B})$ to $\widetilde{\mathcal{B}}$.

For $\alpha>-1$, we define the integral operator  $\mathcal{T} _\alpha $ as follows:
\[
\mathcal{T} _\alpha f(z)=\frac{\varGamma \left( n+1+\alpha \right)}{2^{n+1}\pi ^n\varGamma (\alpha+1)}\int\limits_{T_{B}} \frac{\bfrho(w)^{\alpha}}{\bfrho(z,w)^{n+1+\alpha}} f(w)dV(w),
\]
where $dV(z)$ denotes the Lebesgue measure on $\mathbb{C}^n$.

{\bfseries Theorem D}\label{thm:mainD}
	Suppose that $1\leq p < \infty$, $\lambda>-1$ and $\alpha \in \mathbb{R}$ satisfy
	\[
	\begin{cases}
		\alpha > \frac {\lambda+1}{p}-1,& 1<p<\infty,\\
		\alpha \geq \lambda,& p=1.
	\end{cases}
	\]
	If $f\in A_\lambda^p(T_{B})$ then
	\begin{equation*}\label{eqn:intrepn1}
		f=  \frac{(2i)^N\Gamma(1+\alpha)}{\Gamma(1+\alpha+N)}\mathcal{T} _\alpha  (\bfrho^N \mathcal{L}_n^N f)
	\end{equation*}
	for any $N\in\mathbb{N}_0$.

The structure of this paper is organized as follows: In Section 2, we introduce key terminology and preliminary results. Section 3 is dedicated to the first two theorems, namely Theorem A and Theorem B.  In Section 4, we focus on establishing the last two theorems, Theorem C and Theorem D.

Additionally, the notation $A \lesssim B$ means that there is a positive constant $C$ such that $A \leq C B$, and the notation $A \simeq B$ means that both $A \lesssim B$ and $B \lesssim A$ hold.

\section{Preliminaries and Auxiliary Results}
\subsection{Preliminaries}
	  Let $\mathbb{C}^n$ be the $n$-dimensional complex Euclidean space. For any two points $z=\left( z_1,\cdots ,z_n \right) $ and $\bar{w}=\left( \bar{w}_1,\cdots ,\bar{w}_n \right) $ in $\mathbb{C}^n,$ we write
	  \[z\cdot \bar{w}:=z_1\bar{w}_1+\cdots +z_n\bar{w}_n,\]  \[|z'|^2:=|z_{1}|^{2}+|z_{2}|^{2}+\cdots +|z_{n-1}|^{2},\] 
	  where $z'=(z_1,\cdots, z_{n-1})$.
	  
 The set \[T_B=\left\{ z=x+iy,x\in \mathbb{R}^n,y\in B\subset \mathbb{R}^n \right\} \]is a tubular domain in an n-dimensional complex space $\mathbb{C}^n$, where
	\[B:=\left\{ y=(y',y_n): \left|{y'}\right|^2<y_n, y'=\left( y_1,y_2,\cdots ,y_{n-1} \right)\in \mathbb{R}^{n-1}\right\}.\]
	
	We define the space $L^{p}_\alpha\left( T_B \right)$, which consists of all Lebesgue measurable functions $f$ on $T_B$, with the norm given by:
	 \[\lVert f \rVert _{L^{p}_\alpha\left( T_B \right)}=\left\{ \int_{T_B}{\left| f\left( z \right) \right|^p dV_\alpha\left( z \right)} \right\} ^{\frac{1}{p}}<\infty, \]
	  where  $dV_\alpha(z)=( y_n-\left| y' \right|^2)^\alpha dV(z)$, $\alpha>-1$, and $dV(z)$ denotes the Lebesgue measure on $\mathbb{C}^n$. 
	  
	  Let $A^{p}_\alpha \left( T_B \right)$ denote the weighted Bergman space on $T_B.$ Since the valuation functional is bounded, so $A^{p}_\alpha\left( T_B \right)$ is the closed subspace of $L^{p}_\alpha\left( T_B \right)$. It is easy to know that,  when $1\leqslant p<\infty $, the space $A^{p}_\alpha\left( T_B \right)$ is a Banach space with the norm $\lVert \cdot \rVert _{p}.$ Specially, when $p=2,$ $A^{2}_\alpha\left( T_B \right)$ is a Hilbert space. 
	  
	  The orthogonal projection from $L^{2}_\alpha\left( T_B \right)$ to $A^{2}_\alpha\left( T_B \right)$ is the following integral operator:\[P_\alpha f\left( z \right) =\int_{T_B}{K_\alpha\left( z,w \right) f\left( w \right) dV_\alpha\left( w \right)},\] with the Bergman kernel (\cite{LD} )
	  \[
	  K_\alpha\left( z,w \right) =\frac{2^{n+1+2\alpha}\varGamma \left( n+1+\alpha \right)}{ \varGamma \left( 1+\alpha \right)\pi ^n}\left( \left( z'-\overline{w'} \right) ^2-2i\left( z_n-\overline{w}_n \right) \right) ^{-n -1-\alpha}.
  \] 
	  
	   For convenience, we introduce the following notation:
	  \[\bfrho \left( z,w \right) =\frac{1}{4}\left( \left( z'-\overline{w'} \right) ^2-2i\left( z_n-\overline{w_n} \right) \right)\] and let $\bfrho \left( z \right) :=\bfrho \left( z,z \right) =y_n-|y'|^2.$ 
	  Then the  Bergman kernel of $T_B$ is
	  \[K_\alpha\left( z,w \right) =\frac{\varGamma \left( n+1+\alpha \right)}{2^{n+1}\pi ^n\varGamma (\alpha+1)\bfrho \left( z,w \right) ^{n +1+\alpha}}.\]
	
	Moreover, for $\alpha$, $\lambda>-1$, let $1 \leq  p < \infty $, $P_\alpha$ can be extend to $L^p_\lambda(T_B)$, and $P_\alpha$ is a bounded projection from $L^p_\lambda(T_B)$ onto $A^p_\lambda(T_B)$ if and only if $p(\alpha+1)>\lambda+1$, see \cite[Lemma 3.4]{Li}.
	
For $f\in L^p_\alpha(T_B)$, let $M_f$ denote the multiplication operator induced $f$, the Hankel operator with symbol $f$ is denote by
\[
H^\alpha_f=(\mathbf{I}-P_\alpha)M_fP_\alpha,
\]
 where $\mathbf{I}$ is the identity operator. 
 
 Let  $bT_B:=\left\{ z\in \mathbb{C}^n\,\,: \bfrho \left( z \right) =0 \right\}$ denote the boundary of $T_B.$ Then $\widehat{T_B}:=T_B\cup bT_B\cup \left\{ \infty \right\}$ is the one-point compactification of $T_B.$ Also, let $\partial \widehat{T_B}:=bT_B\cup \left\{ \infty \right\} .$ Thus, $z\rightarrow \partial \widehat{T_B}$ 
means $\bfrho \left( z \right) \rightarrow 0$ or $\left| z \right|\rightarrow \infty .$ 

	Let $\mathcal{M}_+$ be the set of all positive Borel measure $\mu$ such that \[\int_{T_B}{\frac{d\mu \left( z \right)}{\left| \bfrho \left( z,\mathbf{i} \right) \right|^t}}<\infty ,\] for some $t>0$. 

   And let the complex matrix 
	\begin{equation*}
		\mathbf{B}(z):=(b_{ij}(z))_{1\leq i,j\leq n}=\frac{1}{n+1+\alpha}\left ( \frac{\partial ^2}{\partial \bar{z}_i\partial z_j}\ln{K_\alpha(z,z)}\right )_{1\leq i,j\leq n}
	\end{equation*}
	be the Bergman matrix of $T_{B}.$ For a $C^1$ curve $\gamma \,\,: \left[ 0,1 \right] \rightarrow T_B$ , we define 
	\[
	l(\gamma )=\int_{0}^{1}\left \langle \mathbf{B}(\gamma (t))\gamma '(t),\gamma '(t)\right \rangle dt,
	\]
	and call $\beta $ the Bergman metric on $T_{B}$, where 
	$$\beta (z,w)=\inf \{l(\gamma ):\gamma (0)=1,\gamma (1)=w\}.$$
	 Let $D\left( z,r \right) $ denote the Bergman metric ball at $z$ with radius $r$. Thus \[D(z,r)=\left\{w \in T_B:  \beta(z,w)<r\right\}.\]

	Let $|D(z,r)|=V_\alpha(D(z,r))$. For a locally integrable function $f$ on $T_B$, we define a function $\widehat{f}_r$ on $T_B$ as follows:
	\[
	\widehat{f}_r(z)=\frac{1}{|D(z,r)|} \int_{D(z,r)} f(w) dV_\alpha(w),
	\]
	and $\widehat{f}_r(z)$ is the integral mean of $f$ over $D(z,r)$.
	For fixed $r>0$ and $f\in L^p_\alpha(T_{B})$, $1\leq p<\infty$, we define the mean oscillation of $f$ at $z$ in the Bergman metric as follows: 
	\[
	MO_r(f)(z)= \left( \frac{1}{|D(z,r)|} \int_{D(z,r)} |f(w)-\widehat{f}_r(z)|^p dV_\alpha(w) \right)^{1/p}.
	\]
 The space $BMO^p_r$ consists of those functions $f\in L^p_\alpha(T_{B})$ such that
	\[
	\|f\|_{BMO^p_r}= \sup\left\{MO_r(f)(z):z\in T_{B}\right\}<\infty.
	\]
We denote by $C_0\left( T_B \right) $ the space of complex-valued continuous functions $f$ on $T_B$ such that $f\left( z \right) \rightarrow 0$ as $z\rightarrow \partial \widehat{T_B}.$
We say that $f\in VMO_r$ if $MO_r(f)\in C_0(T_{B})$.
	
	For a continuous function $f$ on $T_B$, any $r>0$, let 
	\[
	\omega _r(f)(z)=\sup\{\left | f(z)-f(w)\right |:w\in D(z,r)\}.
	\]
	The function $\omega _r(f)(z)$ is called the oscillation of $f$ at $z$ in the Bergman metric. Let $BO_r$ denote the space of continuous functions $f$ such that 
	\[
	\left \| f\right \|_{BO_r}=\sup\{\omega _r(f)(z):z\in T_B\}<\infty .
	\]
	
		Recall that the Berezin transform over $T_{B}$ is denoted by
	\[
	B_\alpha{f}(z) = \int_{T_{B}} f(w)|k^\alpha_z(w)|^2 dV_\alpha(w),\quad z\in T_{B},
	\]
	where
	\[
	k^\alpha_z(w) = K_\alpha(w,z)/\sqrt{K_\alpha(z,z)},\quad w\in T_{B}.
	\]
	Actually, the Berezin transform is bounded on $L^p_\alpha(T_{B})$ for $1\leq p<\infty$, see \cite{LD}.

Let $BA^p_r$ denote the spaces of all functions on $T_B$ with the property that 
$\widehat{|f|^p_r}(z)\in L^\infty (T_B)$.
From \cite[Theorem 6.1]{Li}, we know that $\widehat{|f|^p_r}(z)$ is bounded if and only if $B_\alpha(\left | f\right |^p)(z)$ is bounded, so which means $BA^p_r$ is independent of $r$, simply write $BA^p$ for $BA^p_r$.

\subsection{Cayley transform and automorphism}

We will use the transform $\varPhi :\mathbb{B}_{n}\rightarrow T_{B}$ as follows, see \cite{LD}
\begin{equation}\label{cayley}
\varPhi(z)=\left ( \frac{\sqrt{2}z'}{1+z_{n}}, i\frac{1-z_n}{1+z_n}-i\frac{z' \cdot z'}{(1+z_n)^2} \right ) ,\ \ z\in \mathbb{B}_n .
\end{equation}
And it is not hard to calculate that 
$$\varPhi^{-1}(w)=\left ( \frac{2iw'}{i+w_{n}+\frac{i}{2}w'\cdot w'}, \frac{i-w_n-\frac{i}{2}w'\cdot w'}{i+w_{n}+\frac{i}{2}w'\cdot w'} \right ) ,\ \ w\in T_{B} .$$

For each fixed $z\in T_B$, we give the holomorphic automorphism of $T_{B}$ as follows:
$$h_z\left( u \right) :=\left( u'-z',u_n-\text{Re}z_n-iu'\overline{z'}+\frac{i\left| z' \right|^2}{2}+\frac{i\overline{z'}\cdot \overline{z'}}{4}+\frac{iu'\cdot z'}{4} \right). $$
Obviously, $h_z(u)$ is a holomorphic automorphism of $T_{B}.$ Hence, the mapping $\sigma _{z}:=\delta _{\rho (z)^{-\frac{1}{2}}}\circ h_{z}$ is a holomorphic automorphism of $T_{B}.$
Simple calculations show that $\sigma _{z}(z)=\mathbf{i}:=(0',i)$ and 
\begin{equation}\label{eqn:jacobian}
	(J_{\mathbb{C}}\sigma_{z})(u)=\bfrho (z)^{-\frac{n+1}{2}},
\end{equation}
where $(J_{\mathbb{C}}\sigma_{z})(u)$ stands for the complex Jacobian of $\sigma_{z}$ at $u.$

Through the Cayley transform, a class of M\"obius transformations of $T_{B}$ induced by $\varphi_{\xi}$ is given by
\begin{equation}\label{transform}
	\tau_z:=\Phi\circ \varphi_{\Phi^{-1}(z)}\circ\Phi^{-1}.
\end{equation}
 Obviously $\Phi(0)=\mathbf{i}$, $\tau_z(z)=\mathbf{i}$. 

To examine the properties of holomorphic functions on tubular domains, it is crucial to understand the behavior of holomorphic mappings from the unit ball to these domains. The following key lemma, adapted from \cite{LD}, highlights essential properties of such mappings: 
	\begin{lemma}\label{auto}
		The following properties hold for holomorphic mapping $\varPhi$ from $\mathbb{B}_n$ to $T_B$:	
	\begin{enumerate}
		\item[(i)] The real Jacobian of $\varPhi$ at $z\in T_{B}$ is 
		\begin{equation}
			J_{R}(\varPhi(z))=\frac{2^{n+1}}{{\left | 1+z_n\right |}^{2(n+1)}}.
		\end{equation}
		
		\item[(ii)] The real Jacobian of $\varPhi^{-1}$ at $z\in T_{B}$ is 
		\begin{equation}\label{eqn:jacobian Phi^-1} 
			(J_{R} \varPhi^{-1})(z)=\frac{1}{4\left| \bfrho (z,\mathbf{i} )\right|^{2(n+1)}} .
		\end{equation}
		
		\item[(iii)]  The identity 
		\begin{equation}
			1-\left \langle \varPhi^{-1}(z),\varPhi^{-1}(w)\right \rangle= \frac{\bfrho (z,w)}{\bfrho (z,\mathbf{i} )\bfrho (\mathbf{i},w)}
		\end{equation}
		holds for all $z,w\in T_{B}$, where $\mathbf{i}=(0',i).$ \\ And moreover, 
		$$1- {\left |\varPhi^{-1}(z)\right |}^2= \frac{\bfrho (z)}{{\left | \bfrho (z,\mathbf{i})\right |}^2}, 1+[\varPhi^{-1}(z)]_{n} = \frac{1}{\bfrho (z,\mathbf{i})}.$$
		
		\item[(iv)] The identity 
		\begin{equation}
			\bfrho (z,w)= \bfrho (\varPhi (\xi ),\varPhi (\eta ))= \frac{1-\left \langle \xi ,\eta \right \rangle}{(1+\xi_{n})(1+\eta_n)}
		\end{equation}
		holds for all $z,w\in T_{B},$ where $\xi =\varPhi^{-1}(z), \eta =\varPhi^{-1}(w).$
		\item[(v)] 	For any $z\in T_B$, $\alpha>-1$ and $r>0$ we have
		\begin{equation}\label{D(z,r)}
			V_\alpha\left( D\left( z,r \right) \right) \simeq \bfrho \left( z \right) ^{n+1+\alpha}.
		\end{equation}
		
	\end{enumerate}
\end{lemma}

\subsection{The Bloch space}
In this section, we introduce the definition of the Bloch space on $T_{B}$. It is important to note that, while the Siegel upper half-space and the unit ball are biholomorphically equivalent, a clear and straightforward definition of the Bloch space on $T_{B}$ has not been previously provided. To ensure that the Bloch space possesses Möbius invariance, we first introduce the invariant gradient.

For $z\in T_B$, let $(b_{i,j})$ denote an $n\times n$ Hermitian matrix and $(b^{i,j})$ its inverse.
Set $b(z)=\det (b_{i,j}(z))$. Define
$$B(z)=\left ( b_{i,j}(z)\right )=\frac{1}{n+1+\alpha}\left ( \frac{\partial ^2}{\partial \overline{z_i}\partial z_j}\ln{K_\alpha(z,z)}\right ), 1\leq i,j\leq n .$$
Let $Y=(y_1,y_2,\cdots, y_{n-1},-\frac{1}{2})$, the matrix $I_{n-1}$ is an $n-1$ dimensional identity matrix, and
\[
A(y)=Y'Y, \, \,   \, I'=\begin{pmatrix}I_{n-1}
	& 0\\ 
	0& 0
\end{pmatrix}.
\]
Then
\begin{equation*}
	\begin{aligned}
		\left ( {b}_{i,j}(z)\right )=\frac{\frac{1}{2}(y_n-{|y'|}^2)I'+A(y)}{(y_n-{|y'|}^2)^2}=\frac{1}{\mathbf{\bfrho }(z)^2}
		\begin{pmatrix}
			\frac{\mathbf{\bfrho }(z)}{2}I_{n-1}+{y'}^Ty'& -\frac{1}{2}{y'}^T\\
			-\frac{1}{2}y'& \frac{1}{4}
		\end{pmatrix}, 
	\end{aligned}
\end{equation*}
and we can get the inverse matrix
\begin{equation*}
	\left ( {b}^{i,j}(z)\right )=\mathbf{\bfrho }(z)
	\begin{pmatrix}
		2I_{n-1}& 4{y'}^T\\
		4y'& 4(y_n+{|y'|}^2)
	\end{pmatrix}, \, \,  b(z)=\frac{1}{(2\bfrho(z))^{n+1}}.
\end{equation*}

From \cite{Sto199}, it is well known that the Laplace-Beltrami operator associated with the Bergman kernel $K$ is the differential operator $\widetilde{\Delta}$ defined by
\[
\widetilde{\Delta}  =\frac{2}{b} \sum_{i,j}\bigg\{\frac{\partial}{\partial \bar{z}_i}\bigg(b b^{i,j}\frac{\partial}{\partial z_j}\big)+\frac{\partial}{\partial z_j}\bigg(bb ^{i,j}\frac{\partial}{\partial \bar{z}_i}\bigg)\bigg\}.
\]
Then substituting the above into the operator, we have
\begin{equation}\label{laplace}
	\begin{split}
		\widetilde{\Delta} = 8\bfrho(z) \Bigg\{ & \sum_{j=1}^{n-1} \frac{\partial^2}{\partial z_j \partial \bar{z}_j} 
		+ \sum_{j=1}^{n-1} 2y_j \frac{\partial^2}{\partial z_n \partial \bar{z}_j} \\
		& + \sum_{j=1}^{n-1} 2y_j \frac{\partial^2}{\partial \bar{z}_n \partial z_j} 
		+ 2(y_n + |y'|^2) \frac{\partial^2}{\partial \bar{z}_n \partial z_n} \Bigg\}
	\end{split}
\end{equation}

The operator $\widetilde{\Delta}$ is often referred to the invariant Laplacian, see \cite{Flensted}, since it has the following property.
\begin{lemma}
	\[
	\widetilde{\Delta}(g\circ \psi)=(\widetilde{\Delta}g)\circ\psi \quad \text{for all} \quad \psi\in \mathrm{Aut}(T_B),
	\]
	where $g$ is $C^2(T_B)$ and $\mathrm{Aut}(T_B)$ denote the group of automorphisms of $T_B$.
\end{lemma}
\begin{proof}
	From cayley transform, so we let
	\[
	\varPhi(z)=(w_1, w_2, \cdots, w_n)=\left ( \frac{\sqrt{2}z'}{1+z_{n}}, i\frac{1-z_n}{1+z_n}-i\frac{z' \cdot z'}{(1+z_n)^2} \right ) ,\ \ z\in \mathbb{B}_n.
	\]
	Though \eqref{transform}, let $\psi	=\Phi\circ \varphi_{\Phi^{-1}(z)}\circ\Phi^{-1}$, and $\psi(z)=\mathbf{i}$.
	For the unit ball $\mathbb{B}_n$, the Laplace-Beltrami operator denote by $\widetilde{\Delta}_{\mathbb{B}_n}$,  then from \cite{Sto199}, we have
	\[
	\widetilde{\Delta}_{\mathbb{B}_n}=4(1 - |z|^2)\sum_{i,j=1}^{n} \left[\delta_{i,j} - \bar{z}_i z_j\right] \frac{\partial^2}{\partial z_j \partial \bar{z}_i},
	\]
	where $\delta_{i,j}$ is the Kronecker delta.
	And
	\[
	\widetilde{\Delta}_{\mathbb{B}_n}(f\circ \varphi)=(\widetilde{\Delta}_{\mathbb{B}_n}f)\circ \varphi, 
	\]
		where $f$ is $C^2(\mathbb{B}_n)$, and $\varphi \in \mathrm{Aut}(\mathbb{B}_n)$.

By $\varPhi$ is holomorphic mapping, means $\frac{\partial w_i}{\partial \bar{z}_j}=0$, then using the following complex version of the chain rule (see \cite{Rudin1980}):
\begin{equation*}
	\frac{\partial g}{\partial z_j}=\sum_{i=1}^{n} \left ( \frac{\partial f}{\partial w_i}\frac{\partial w_i}{\partial z_j} +\frac{\partial f}{\partial \bar {w}_i}\frac{\partial \bar {w}_i}{\partial z_j}  \right ) ,
\end{equation*}
\begin{equation*}
	\frac{\partial g}{\partial \bar{z}_j}=\sum_{i=1}^{n} \left ( \frac{\partial f}{\partial w_i}\frac{\partial w_i}{\partial \bar{z}_j} +\frac{\partial f}{\partial \bar {w}_i}\frac{\partial \bar {w}_i}{\partial \bar{z}_j}  \right ) .
\end{equation*}

 For any $ j=0,\cdots ,n-1$, we have
\[
\frac{\partial (f\circ \varPhi)}{\partial \bar {z}_j}=\sqrt{2} \frac{\partial  f(\mathbf{i} )}{\partial  \bar{w}_j } , 
\]
\[
 \frac{\partial^2(f\circ \varPhi)(0)}{\partial z_j\partial \bar{z}_j } =2 \frac{\partial^2f(\mathbf{i} )}{\partial {w}_j\partial \bar{w}_j},
\]
and
\[
\frac{\partial (f\circ \varPhi)}{\partial \bar {z}_n}=2i \frac{\partial  f(\mathbf{i} )}{\partial  \bar{w}_n },
\]
\[
 \frac{\partial^2(f\circ \varPhi)(0)}{\partial z_n\partial \bar{z}_n } =4\frac{\partial^2f(\mathbf{i} )}{\partial {w}_n\partial \bar{w}_n}.
\]
Then we have
\[
(\widetilde{\Delta}g)(\mathbf{i})=\widetilde{\Delta}_{\mathbb{B}_n}(g\circ \varPhi)(0).
\]
Moreover, for $w\in \mathbb{B}_n$, we can choose $\varphi_w\in \mathrm{Aut}(\mathbb{B}_n)$, $\varphi_w(0)=w$, and $\varPhi \circ\varphi_w$ is still a holomorphic mapping from $\mathbb{B}_n$ to $T_B$, so
\begin{align*}
	\widetilde{\Delta}_{\mathbb{B}_n}(g\circ \varPhi)(w)
	&=\widetilde{\Delta}_{\mathbb{B}_n}(g\circ \varPhi)(\varphi_w(0))\\
	&=\widetilde{\Delta}_{\mathbb{B}_n}(g\circ \varPhi \circ \varphi_w)(0)\\
	&=(\widetilde{\Delta}g)\circ (\varPhi \circ \varphi_w(0))\\
	&=(\widetilde{\Delta}g)\circ (\varPhi(w)).
\end{align*}
Hence, for $z\in T_B$, we have
\begin{align*}
(\widetilde{\Delta}g)\circ\psi(z)
=(\widetilde{\Delta}g)(\mathbf{i})&=\widetilde{\Delta}_{\mathbb{B}_n}(g\circ \varPhi)(0)\\
&=\widetilde{\Delta}_{\mathbb{B}_n}(g \circ \varPhi) \circ \varphi_{\Phi^{-1}(z)}(\Phi^{-1}(z))\\
&=\widetilde{\Delta}_{\mathbb{B}_n}(g \circ \varPhi \circ \varphi_{\Phi^{-1}(z)})\circ (\Phi^{-1}(z))\\
&=\widetilde{\Delta}(g \circ \varPhi \circ \varphi_{\Phi^{-1}(z)}\circ \Phi^{-1})(z)\\
&=\widetilde{\Delta}(g \circ\psi)(z).
\end{align*}
\end{proof}

If $u$ and $v$ are $C^2(T_B)$ functions, then from \cite{Sto199},
\[
\widetilde{\Delta}(uv)=u\widetilde{\Delta}v+2(\widetilde{\nabla}u)v+v\widetilde{\Delta}u,
\]
where $\widetilde{\nabla}u$ is the vector field defined by 
\[
\widetilde{\nabla}u=2\sum_{i,j} b^{i,j}\Bigg\{\frac{\partial u}{\partial \bar{z}_i} \frac{\partial}{\partial z_j}+
\frac{\partial u}{\partial z_j} \frac{\partial}{\partial \bar{z}_i} \Bigg\}.
\]
Furthermore, if $f$ in $H(T_B)$, it's not hard to check that 
\[
\widetilde{\Delta} |f|^2=2 (\widetilde{\nabla} f)\overline{f}=2|\widetilde{\nabla} f|^2.
\]
In particular, we have
\[
|\widetilde{\nabla}f|^2=2\sum_{i,j} b^{i,j} \overline{\frac{\partial f}{\partial z_i}} \frac{\partial f}{\partial z_j} .
\]

Hence, from  the invariance property of $\widetilde{\Delta}$, we have
\[
|\widetilde{\nabla}(f\circ\psi)|=|\widetilde{\nabla}f\circ\psi|\quad \text{for all} \quad \psi\in \mathrm{Aut}(T_B).
\]
 Therefore, the operator $|\widetilde{\nabla}|$ is usually called invariant gradient of $T_B$.

With this in mind, we will use the invariant gradient to directly define the Bloch space on $T_{B}$. The form of the invariant gradient is given by:
$$\left |\widetilde{\nabla }f(z) \right |^2=4\mathbf{\bfrho }(z) \left ( 2\mathbf{\bfrho }(z) \left | \frac{\partial f(z)}{\partial z_n}\right |^2 +\displaystyle\sum_{j=1}^{n-1}\left | \frac{\partial f(z)}{\partial z_j}+2y_j\frac{\partial f(z)}{\partial z_n}\right |^2\right )$$
for $f\in H(T_B)$.

We define the Bloch space of $T_B$, denoted by $\mathcal{B}$, using invariant gradients. It is the space of functions $f\in H(T_B)$ such that 
\[\|f\|_{\mathcal{B}}:=\sup\{|\widetilde{\nabla} f(z)|:z\in T_B\}<\infty.\]
Clearly, this defines only a semi-norm, which is invariant under the action of $\text{Aut}(T_B)$. 
The little Bloch space, denoted by $\mathcal{B}_0$, is defined as follows:
$$\mathcal{B}_0:=\{f\in \mathcal{B}:\widetilde{\nabla} f\in C_0(\widehat{T_B})\},$$
where $C_0(\widehat{T_B})=\{f\in C(\widehat{T_B}): {\lim_{z\to \partial\widehat{T_B}}}\left | f(z)\right |=0\}$.

Next, we define the differential operator as follows:
$$\mathcal{L}_n:=\frac{\partial }{\partial z_n}, \ \ \mathcal{L}_j:=\frac{\partial }{\partial z_j}+2y_j\frac{\partial }{\partial z_n}, \ \ j=1,\cdots,n-1 .$$
For $\gamma\in \mathbb{N}^n_0$, which $\mathbb{N}_0$ denotes the set of non-negative integers, we write
\[
\mathcal{L}^\gamma :=(\mathcal{L}_1)^{\gamma_1}\cdots (\mathcal{L}_n)^{\gamma_n}.
\]
The properties of the Bloch space will be established using the differential operator introduced earlier.

 We will consider a new space 
$$\widetilde{\mathcal{B}}:=\left \{f\in \mathcal{B}:f(\mathbf{i})=0 \right \},$$
where $\mathbf{i}:=(0',i)$.

Using  $\tau_z$, we establish the relationship between the invariant gradients on $T_{B}$ and $\mathbb{B}$ as follows.
For any $f\in H(T_{B})$, it is straightforward to verify that:
\begin{equation}\label{eq:tildenablaf(i)}
	|\widetilde{\nabla} f(\mathbf{i})|=\sqrt{2}|\nabla(f\circ\Phi)(0)|,
\end{equation}
then we have
\begin{align}
	|\widetilde{\nabla}f(z)|&= |\widetilde{\nabla}(f\circ\tau_z)(\mathbf{i})| = \sqrt{2}|\nabla(f\circ\tau_{z}\circ\Phi)(0)|\nonumber\\
	&=\sqrt{2}|\nabla(f\circ\Phi\circ\varphi_{\Phi^{-1}(z)})(0)|\nonumber\\
	&=\sqrt{2}|\widetilde{\nabla}_{\mathbb{B}}(f\circ\Phi)\big(\Phi^{-1}(z)\big)|\label{eq:gradient}
\end{align}
for all $z\in T_{B}$.

\subsection{Some auxiliary results}

In order to establish our main theorems, we provide several auxiliary results that will support our proofs. We will omit most of the proofs for brevity.
\begin{lemma}[{\cite[Therom 2]{LD}}]
	There exists a positive integer $N$ such that for any $0<r\leqslant1$ we can find a sequence $\{a_k\}$ in $T_B$ with the following properites:
	\begin{enumerate}
		\item $T_B=\bigcup_{k=1}^{\infty}{D\left( a_k,r \right)};$
		\item The sets $D(a_k,r/4)$ are mutually disjoint;
		\item Each point $z\in T_B$ belongs to at most $N$ of the sets $D(a_k,r)$.
	\end{enumerate}
\end{lemma}

\begin{lemma}[{\cite[Lemma 2]{LD}}]\label{eqn:keylem2}
	Let $r, s>0, t>-1,$ and $r+s-t>n+1$, then 
	\begin{equation}
		\int_{T_{B}}\frac{{\bfrho (w)}^{t}}{{\bfrho (z,w)}^{r}{\bfrho (w,u)}^{s}}dV(w)=\frac{C_{1}(n,r,s,t)}{\bfrho (z,u)^{r+s-t-n-1}}
	\end{equation}
	for all $z,u \in T_{B}$, where
	\[
C_{1}(n,r,s,t)=\frac{2^{n+1}{\pi}^{n}\Gamma (1+t)\Gamma (r+s-t-n-1)}{\Gamma (r)\Gamma (s)}.
	\]
	In particular, let $s,t\in \mathbb{R}$, if $t>-1,s-t>n+1$, then \[\int_{T_B}{\frac{\bfrho \left( w \right) ^t}{\left| \bfrho \left( z,w \right) \right|^s}}dV\left( w \right) =\frac{C_1\left( n,s,t \right)}{\bfrho \left( z \right) ^{s-t-n-1}}.\]
\end{lemma}

\begin{lemma}[{\cite[Theorem 2.2]{DK}}]\label{projection}
	Suppose that $1\leq p < \infty$, $\lambda>-1$ and $\alpha\in \mathbb{R}$ satisfies
	\[
	\begin{cases}
		\alpha > \frac {\lambda+1}{p}-1,& 1<p<\infty,\\
		\alpha \geq \lambda,& p=1.
	\end{cases}
	\]
	If $f\in A^p_\lambda (T_{B})$, then $f=\mathcal{T} _\alpha f$.
\end{lemma}

\begin{lemma}\label{equalivent}
	For any $r>0$, the inequalities 
	$$\left| \bfrho \left( z,u \right) \right|\simeq \left| \bfrho \left( z,v \right) \right|$$
	hold for all $z,u,v\in T_{B} $ with $\beta (u,v) <r .$
\end{lemma}

\begin{lemma}[{\cite[Lemma 3.7]{Li}}]
	For any $z,w\in {T_{B}}$, we have 
	\begin{equation}\label{essiential}
	2\left| \bfrho \left( z,w \right) \right|\ge \max \left\{ \bfrho \left( z \right) ,\bfrho \left( w \right) \right\}.		
	\end{equation}

\end{lemma}

\begin{lemma}\label{chara}
Let $1\leq p<\infty $, $r>0$, and $f\in L^p_\alpha(T_B)$, then $f\in BMO^p_r$ if and only if there is a constant $\lambda_z$ such that, for all $z\in T_B$, with
\[
\frac{1}{\left | D(z,r)\right |}\int_{D(z,r)}\left | f(w)-\lambda_z \right |^p dV_\alpha(w)\leq C.
\] 
\end{lemma}
\begin{proof} 
 If $f\in BMO^p_r$, then the inequality holds when $\lambda_z=\widehat{f}_r(z)$. 
 
 Conversely, assume that the above inequality holds for all $z\in T_B$, so by the triangle inequality, 
 and Holder's inequality, we have
 \begin{align*}
 	&\left ( \frac{1}{\left | D(z,r)\right |}\int_{D(z,r)}\left | f(w)-\widehat{f}_r(z)\right |^p dV_\alpha(w)\right )^{1/p}\\
 	&\leq 
 	\left (\frac{1}{\left | D(z,r)\right |}\int_{D(z,r)}\left | f(w)-\lambda_z \right |^p dV_\alpha(w)\right )^{1/p}+\left | \widehat{f}_r(z)-\lambda_z\right |\\
 	&=\left (\frac{1}{\left | D(z,r)\right |}\int_{D(z,r)}\left | f(w)-\lambda_z \right |^p dV_\alpha(w)\right )^{1/p}
 	+\left | \frac{1}{\left | D(z,r)\right |}\int_{D(z,r)}\left ( f(w)-\lambda_z \right ) dV_\alpha(w)\right |\\
 	&\leq 2\left (\frac{1}{\left | D(z,r)\right |}\int_{D(z,r)}\left | f(w)-\lambda_z \right |^p dV_\alpha(w)\right )^{1/p}\\
 	&\leq 2C.
 \end{align*}
\end{proof}

\begin{lemma}
	Let \( r > 0 \) and \( f \) be a continuous function on \( T_B \). Then \( f \in BO_r \) if and only if there exists a constant \( C \) such that
	\[
	|f(z) - f(w)| \leq C (\beta(z,w) + 1)
	\]
	for all \( z, w \in T_B \). Furthermore, \( BO_r \) is independent of \( r \), so we will simply write \( BO \) for \( BO_r \).
\end{lemma}

\begin{proof}
	It is similar with \cite[Lemma 3.3]{Si}, so we omit it.
\end{proof}

In particular, we also have the following result:
\begin{cor}\label{cor:Blochbeta}
	If \( f \in \mathcal{B} \), then
	\[
	|f(z) - f(w)| \leq \frac{\|f\|_{\mathcal{B}} \beta(z,w)}{\sqrt{2}}
	\]
	for all \( z, w \in T_B \).
\end{cor}

\section{Proofs of the Main Results: Theorems A and B}
In this section, we provide the proofs of our main results, {\bfseries Theorems A and B}, which characterize the relationships between $BMO^p_r$ and properties of holomorphic functions. {\bfseries Theorem A} states that the following conditions are equivalent:
\begin{thm}
Suppose $r>0$, $1\leq p<\infty $, and $p(\alpha+1)>\lambda+1$, let $f\in L^p_\lambda(T_B)$, the following conditions are equivalent:
\begin{enumerate}[(a)]
	\item $f\in BMO^p_r$.
	\item $f=f_1+f_2$ with $f_1\in BO$ and $f_2\in BA^p$.
	\item $H^{(\alpha)}_f$ and $H^{(\alpha)}_{\overline f}$ are both bounded on $A^p_\lambda(T_B)$.
\end{enumerate}
\end{thm}

\begin{proof}
	 $(a) \Rightarrow (b)$: It suffices to show that $BMO^p_r \subset BO+BA^p$. Given $z, w\in T_B$ with $\beta(z,w)<r$, we have

\begin{align*}
	\left | \widehat{f}_r(z)-\widehat{f}_r(w)\right |
	&\leq \left | \widehat{f}_r(z)-\widehat{f}_{2r}(z)\right |+\left | \widehat{f}_{2r}(z)-\widehat{f}_r(w)\right |\\
	&\leq \frac{1}{\left | D(z,r)\right |}\int_{D(z,r)}\left | f(u)-\widehat{f}_{2r}(z)\right |dV_\lambda(u)\\&+\frac{1}{\left | D(w,r)\right |}\int_{D(w,r)}\left | f(v)-\widehat{f}_{2r}(z)\right |dV_\lambda(v).
\end{align*}

By Holder's inequality, and Lemma \ref{equalivent}, so
\[
\left | \widehat{f}_r(z)-\widehat{f}_r(w)\right |\leq C \left \| f\right \|_{BMO^p_{2r}}.
\]
It is straightforward to show that $\widehat{f}_r(z)$ is continuous on $T_B$. Therefore, we have  $\widehat{f}_r(z)\in BO$.

Let $g=f-\widehat{f}_r$, we show that $g\in BA^p$. By triangle inequality,

\begin{align*}
	\left [ \widehat{\left | g\right |^p_r}(z)\right ]^{1/p}&=\left( \frac{1}{\left | D(z,r)\right |}\int_{D(z,r)}\left | f(u)-\widehat{f}_{r}(u)\right |^p dV_\lambda(u)\right)^{1/p}\\
	&\leq \left( \frac{1}{\left | D(z,r)\right |}\int_{D(z,r)}\left | f(u)-\widehat{f}_{r}(z)\right |^p dV_\lambda(u)\right)^{1/p}\\
	&+\left( \frac{1}{\left | D(z,r)\right |}\int_{D(z,r)}\left | \widehat{f}_{r}(z)-\widehat{f}_{r}(u)\right |^p dV_\lambda(u)\right)^{1/p}\\
	&\leq \left \| f\right \|_{BMO^p_r}+\omega _r(\widehat{f}_{r})(z)\\
	&\leq C.
\end{align*}
So, $g\in BA^p$, $f=\widehat{f}_{r}+(f-\widehat{f}_{r})\in BO+BA^p$.

$(b) \Rightarrow (c)$: For $f\in BO$, and $g\in A^p_\lambda(T_B)$, from Lemma \ref{projection}, we have
\begin{align*}
	\left | H^{(\alpha)}_fg(z)\right |&=\left | (I-P_\alpha)(fg)(z)\right |\\
	&\leq \int_{T_B}\left | f(z)-f(w)\right ||g(w)K_\alpha(z,w)| dV_\alpha(w)\\
	&\leq C\int_{T_B}\frac{\beta(z,w)+1}{|\bfrho(z,w)|^{n+1+\alpha}}|g(w)|dV_\alpha(w).
\end{align*}
From \cite{Si}, for any $\varepsilon>0$, we know that
\[
\beta(z,w)\leq 2^{2\varepsilon -1}\frac{|\bfrho(z,w)|^{2\varepsilon }}{\bfrho(z)^\varepsilon \bfrho(w)^\varepsilon },
\]
then the operator
\[
Tg(z):=\int_{T_B}\frac{\beta(z,w)+1}{|\bfrho(z,w)|^{n+1+\alpha}}|g(w)|dV_\alpha(w)
\]
is bounded on $A^p_\lambda(T_B)$, see \cite[Lemma 3.4]{Li}.

Next, for $f\in BA^p$, it follows from \cite[Lemma 6.1]{Li}, let $d\mu _f(z)=|f(z)|^pdV_\lambda(z)$.
Given $t \in \mathbb{R}$, we denote by $S_t$ the vector spaces of functions $f$ that are holomorphic in $T_B$ and satisfy the condition: \[\sup_{z\in T_B}\left| \bfrho \left( z,\mathbf{i} \right) \right|^t\left| f\left( z \right) \right|<\infty ,
\]
and it is dense in $A^p_\lambda(T_B)$ when $t>n+1/2$. So, let $g\in S_t$, we have
\begin{align*}
	\left \| H^{(\alpha)}_fg\right \|_{A^p_\lambda (T_B)}&=\left \| fg\right \|_{A^p_\lambda (T_B)}+\left \| P_\alpha(fg)\right \|_{A^p_\lambda (T_B)}\\
	&\lesssim \left \| fg\right \|_{A^p_\lambda (T_B)}
	=\int_{T_B}\left | g(z)\right |^p d\mu _f(z)\\
	&\leq C\int_{T_B}\left | g(z)\right |^p dV_\lambda(z).
\end{align*}
Thus, we have $H^{(\alpha)}_f$ and $H^{(\alpha)}_{\overline f}$ are both bounded on $A^p_\lambda(T_B)$.

$(c)\Rightarrow(a)$: Inspired by \cite{Pau2016}, for $f\in L^p_\lambda(T_B)$, and any $z\in T_B$, we can define the function as follows:
\[
\mathbf{MO}_\lambda f(z):=\left \| fh_z-B_\lambda f(z)h_z\right \|_{L^p_\lambda(T_B)},
\]
and
\[
f_z(w)=\frac{P_\alpha(\overline{f}h_z)(w)}{h_z(w)},\, w\in T_B,
\]
where $h_z(w)=\frac{K_\lambda(z,w)}{\left \| K_\lambda(z,\cdot)\right \|_{A^p_\lambda}}$.

By \eqref{essiential}, it is easy to know that
\[
|K_\lambda(z,w)|\leq 2^{n+1+\lambda}K_\lambda(z,z),
\]
and 
\[
\left \| K_\lambda(z,\cdot)\right \|^p_{A^p_\lambda}=\int_{T_B}\frac{\bfrho(w)^\lambda}{|\bfrho(z,w)|^{p(n+1+\lambda)}}dV(w)\asymp \bfrho(z)^{(1-p)(n+1+\lambda)}.
\]
Thus, let $\frac{1}{p}+\frac{1}{q}=1$, we have
\[
\left ( \int_{T_B}\left | f(w)h_z(w)\right |^pdV_\lambda(w)\right )^{1/p}\leq2^{n+1+\lambda}\left \| K_\lambda(z,\cdot)\right \|^{q-1}_{A^p_\lambda}\left ( \int_{T_B}|f(w)|^pdV_\lambda(w)\right )^{1/p},
\]
which implies that $\overline{f}h_z\in L^p_\lambda(T_B)$. The function $f_z$  and $P_\alpha(\overline{f}h_z)$ are both well-defined, see \cite[Lemma 3.4]{Li}. Then, similarly to \cite[Lemma 4.4]{Si}, from Lemma \ref{projection}, we have $\overline{f_z(z)}=B_\lambda f(z)$ and $P_\alpha(\overline{f_z}h_z)(w)=B_\lambda f(z)h_z(w)$.

And actually,
\begin{align*}
	\mathbf{MO}^p_\lambda f(z)&=\left \| fh_z-B_\lambda f(z)h_z\right \|^p_{L^p_\lambda(T_B)}\\
	&=C\int_{T_B}\frac{\left | f(w)-B_\lambda f(z)\right |^p}{|\bfrho(z,w)|^{p(n+1+\lambda)}}\bfrho(z)^{(p-1)(n+1+\lambda)}dV_\lambda(w)\\
	&\geq C\int_{D(z,r)}\frac{\left | f(w)-B_\lambda f(z)\right |^p}{\bfrho(z)^{n+1+\lambda}}dV_\lambda(w)\\
	&\geq \frac{C}{|D(z,r)|}\int_{D(z,r)}\left | f(w)-B_\alpha f(z)\right |^pdV_\lambda(w),
\end{align*}
by Lemma \ref{chara}, if $\mathbf{MO}_\lambda f$ is finite, then $f\in BMO^p_r$.

Using $P_\alpha(\overline{f_z}h_z)(w)=B_\lambda f(z)h_z(w)$, we have
\begin{align*}
	\mathbf{MO}_\lambda f(z)&=\left \| fh_z-B_\lambda f(z)h_z\right \|_{L^p_\lambda(T_B)}\\
	&=\left \| fh_z-P_\alpha(\overline{f_z}h_z)\right \|_{L^p_\lambda(T_B)}\\
	&\leq \left \| fh_z-P_\alpha(fh_z)\right \|_{L^p_\lambda(T_B)}+\left \| P_\alpha(fh_z)-P_\alpha(\overline{f_z}h_z)\right \|_{L^p_\lambda(T_B)}\\
	&\leq \left \| H^{(\alpha)}_fh_z\right \|_{L^p_\lambda(T_B)}+\left \| P_\alpha\right \|_{L^p_\lambda(T_B)}\left \| \overline{f}h_z-f_zh_z\right \|_{L^p_\lambda(T_B)}\\
	&=\left \| H^{(\alpha)}_fh_z\right \|_{L^p_\lambda(T_B)}+\left \| P_\alpha\right \|_{L^p_\lambda(T_B)}\left \|\overline{f}h_z- P_\alpha(\overline{f}h_z)\right \|_{L^p_\lambda(T_B)}\\
	&=\left \| H^{(\alpha)}_fh_z\right \|_{L^p_\lambda(T_B)}+\left \| P_\alpha\right \|_{L^p_\lambda(T_B)}\left \| H^{(\alpha)}_{\overline{f}}h_z\right \|_{L^p_\lambda(T_B)}.
\end{align*}

So, if both $H^{(\alpha)}_f$ and $H^{(\alpha)}_{\overline{f}}$ are bounded on $A^p_\lambda(T_B)$, then $f\in BMO^p_\alpha$.

\end{proof}


Before we prove the second main result, Theorem B, it is necessary to introduce several key lemmas.

\begin{lemma}\label{lem:ptwsest}
	For any fixed $\alpha^{\prime}\in \mathbb{N}_0^{n-1}$, then we have
	\[
	\big|(z^{\prime}-w^{\prime})^{\alpha^{\prime}}\big| \lesssim |\bfrho(z,w)|^{\frac {|\alpha^{\prime}|}{2}},
	\]
	for all $z,w\in T_{B}$.
\end{lemma}
\begin{proof}
	For each fixed $z\in T_B$, we define the following holomorphic  self-mapping of $T_{B}$:
	$$h_z\left( u \right) :=\left( u'-z',u_n-\text{Re}z_n-iu'\overline{z'}+\frac{i\left| z' \right|^2}{2}+\frac{i\overline{z'}\cdot \overline{z'}}{4}+\frac{iu'\cdot z'}{4} \right). $$
	All these mappings are holomorphic automorphisms of $T_{B}$. 
	In particular, we have
	\[
	h_z(z)= \bfrho(z) \mathbf{i},
	\]
	where $\mathbf{i}=(0^{\prime},i)$.
	Also, an easy calculation shows that
	\[
	\bfrho(h_z(u),h_z(v))=\bfrho(u,v).
	\]
	Note that
	\begin{equation*}\label{eqn:eleinq}
		4|\bfrho(u, s\mathbf{i})|=|{u^{\prime}}^2 -2iu_n+2s| \geq |u^{\prime}|^2, 
	\end{equation*}
	
	\begin{equation*}
		\begin{aligned}
			4|\bfrho(u, s\mathbf{i})|	\geqslant 4\left| \bfrho \left( u,\mathbf{i} \right) \right|\geqslant 4\text{Re}\bfrho \left( u,\mathbf{i} \right) &=\sum_{k=1}^{n-1}{\text{Re}u_{k}^{2}}+2\text{Im}u_n+2
			\\
			&\geqslant \sum_{k=1}^{n-1}{\text{Re}u_{k}^{2}}+2\sum_{k=1}^{n-1}{\left( \text{Im}u_k \right) ^2}+2
			\\
			&=\sum_{k=1}^{n-1}{\left( \text{Re}u_k \right) ^2}+\sum_{k=1}^{n-1}{\left( \text{Im}u_k \right) ^2}+2
			\\
			&=\sum_{k=1}^{n-1}{\left| u_k \right|^2}+2,
		\end{aligned}
	\end{equation*}
	for any $u\in T_{B}$ and any $s>0$.
	Taking $u=h_z(w)$ and $s=\bfrho(z)$ in the above inequality, and associating the previous argument, we get
	\[
	\big|(h_z(w))^{\prime}\big|^2 \lesssim \left|\bfrho(h_z(w), \bfrho(z) \mathbf{i})\right| = \left|\bfrho(h_z(w), h_z(z))\right|= |\bfrho(w,z)|.
	\]
	Consequently,
	\[
	\big|(z^{\prime}-w^{\prime})^{\alpha^{\prime}}\big| \lesssim \big|(z^{\prime}-w^{\prime})\big|^{|\alpha^{\prime}|}
	= \big|(h_z(w))^{\prime}\big|^{|\alpha^{\prime}|} \lesssim |\bfrho(z,w)|^{\frac {|\alpha^{\prime}|}{2}},
	\]
	as desired.
\end{proof}

\begin{lemma}\label{lem:nabla.Zhu 2.24a}
	For $f\in H(T_{B})$, then
	\[
	|f(z)-f(w)|\leq \sup_{u\in\gamma} |\widetilde{\nabla}f(u)| \beta(z,w)/\sqrt{2},
	\]
	where $\gamma$ is a geodesic joining $z$ to $w$ in the Bergman metric.
\end{lemma}
\begin{proof}
	Let $\beta_{\mathbb{B}}(\cdot,\cdot)$ denote the Bergman metric of $\mathbb{B}$.
	From \cite{Tim80I} that 
	\[
	|h(\xi)-h(\eta)| \leq  \sup_{\zeta\in\Gamma} |\widetilde{\nabla}_{\mathbb{B}}h(\zeta)| \beta_{\mathbb{B}}(\xi,\eta)
	\]
	for $h\in H(\mathbb{B})$ and $\xi,\eta\in\mathbb{B}$, where $\Gamma$ is a geodesic joining $\xi$ to $\eta$ with respect to $\beta_{\mathbb{B}}(\cdot,\cdot)$.
	It is known that $\Phi$ sends the geodesic of $\mathbb{B}$ with respect to $\beta_{\mathbb{B}}(\cdot,\cdot)$ to the geodesic of $T_{B}$ with respect to $\beta(\cdot,\cdot)$, and from \cite[Proposition 1.4.15]{Kra01} ,we have the following 
	\begin{equation}\label{eq:Bergmanmetric}
		\beta(z,w)=\beta_{\mathbb{B}}(\Phi^{-1}(z),\Phi^{-1}(w))
	\end{equation}
	for any $z,w\in T_{B}$. Therefore, let $z=\Phi(\xi)$, $w=\Phi(\eta)$, $u=\Phi(\zeta)$ and $\gamma=\Phi(\Gamma)$ be the geodesic joining $z$ to $w$ with respect to $\beta(\cdot,\cdot)$, we have
	\begin{align*}
		|f(z)-f(w)|&=| (f\circ\Phi)(\xi)-(f\circ\Phi)(\eta)|\\
		&\leq \sup_{\zeta\in\Gamma} |\widetilde{\nabla}_{\mathbb{B}}(f\circ\Phi)(\zeta)| \beta_{\mathbb{B}}(\xi,\eta)\\
		&=\sup_{u\in\gamma} \frac{|\widetilde{\nabla}f(u)|}{\sqrt{2}} \beta(z,w),
	\end{align*}
	where the last equality uses \eqref{eq:gradient}, as desired.
\end{proof}

\begin{lemma}\label{lem:nabla.Zhu 2.24}
	Suppose $r>0$, $\alpha>-1$ and $p>0$. There exists a positive constant $C$ such that 
	\[
	|\widetilde{\nabla}f(z)|^p \leq \frac{C}{|D(z,r)|}\int_{D(z,r)} |f(w)|^p dV_\alpha(w)
	\]
	for all $f\in H(T_{B})$ and all $z\in T_{B}$.
\end{lemma}
\begin{proof}
	From \cite[Lemma 2.4]{Zhu05}, for $g\in H(\mathbb{B})$,  there exists a positive constant $C$ depending on $r$ such that 
	\[
	|\nabla g(0)|^p \leq C \int_{D_{\mathbb{B}}(0,r)} |g(\eta)|^p dV_\alpha(\eta).
	\]
	Then, for any $f\in H(T_{B})$, we have 
	\begin{align*}\label{eq:nabla.Zhu. 2.24}
		|\nabla(f\circ\Phi)(0)|^p &\leq C  \int_{D_{\mathbb{B}}(0,r)} |f\circ\Phi(\eta)|^p dV_\alpha(\eta)\\
		&\leq C \int_{D(\mathbf{i},r)}  \frac{|f(w)|^p}{|\bfrho(w,\mathbf{i})|^{2(n+1+\alpha)}} dV_\alpha(w)\\
		&\leq C \int_{D(\mathbf{i},r)} |f(w)|^p  dV_\alpha(w),
	\end{align*}
 the last inequality uses the inequality \eqref{essiential}, so $|\bfrho(w,\mathbf{i})|\geq 1/2$. By \eqref{eq:tildenablaf(i)}, we have
	\[
	|\widetilde{\nabla} f(\mathbf{i})|^p \leq C\int_{D(\mathbf{i},r)} |f(w)|^p  dV_\alpha(w).
	\]
 Replacing $f$ by $f\circ\sigma_z^{-1}$ in the above inequality, we obtain
	\[
	|\widetilde{\nabla}f(z)|^p
	\leq C \int_{D(\mathbf{i},r)} |f\circ\sigma_z^{-1} (w)|^p  dV_\alpha(w)
	= \frac{C}{|D(z,r)|} \int_{D(z,r)} |f(w)|^p  dV_\alpha(w),
	\]
	where the last equality uses \eqref{eqn:jacobian} and  \eqref{D(z,r)}. The proof of this lemma is completed.
\end{proof}

 We state two key lemmas, lemma 3.5 and lemma 3.6, which are crucial for proving {\bfseries Theorem B}. Using Lemma \ref{lem:nabla.Zhu 2.24} and Lemma \ref{lem:nabla.Zhu 2.24a}, as discussed above, these new lemmas can be proven straightforwardly using a similar approach to that in  \cite{Si}. Therefore, we omit the details.
 
\begin{lemma}\label{thm:BlochBMO}
	Suppose $r>0$, $p\geq 1$, $\alpha>-1$ and $f\in H(T_{B})$. Then the following conditions are equivalent:
	\begin{enumerate}
		\item[(a)] $f\in\mathcal{B}$.
		\item[(b)] There exists a positive constant $C$ such that
		\[
		\frac{1}{|D(z,r)|} \int_{D(z,r)} |f(w)-f(z)|^p dV_\alpha(w) \leq C 
		\]
		for all $z\in T_{B}$.
		\item[(c)] There exists a positive constant $C$ such that
		\[
		\frac{1}{|D(z,r)|} \int_{D(z,r)} |f(w)-\widehat{f}_r(z)|^p dV_\alpha(w) \leq C 
		\]
		for all $z\in T_{B}$.
		\item[(d)] There exists a positive constant $C$ with the property that for every $z\in T_{B}$ there is a complex number $c_z$ such that
		\[
		\frac{1}{|D(z,r)|} \int_{D(z,r)} |f(w)-\lambda_z|^p dV_\alpha(w) \leq C.
		\]
	\end{enumerate}
\end{lemma}

\begin{lemma}\label{thm:Bloch_0VMO}
	Suppose $r>0$, $p\geq 1$, $\alpha>-1$ and $f\in H(T_{B})$. Then the following conditions are equivalent:
	\begin{enumerate}
		\item[(a)] $f\in \mathcal{B}_0$.
		\item[(b)] There is 
		\[
		\frac{1}{|D(z,r)|} \int_{D(z,r)} |f(w)-f(z)|^p dV_\alpha(w) \to 0 \quad \text{as} \quad z\to\partial\widehat{T_{B}}.
		\]
		\item[(c)] There is
		\[
		\frac{1}{|D(z,r)|} \int_{D(z,r)} |f(w)-\widehat{f}_r(z)|^p dV_\alpha(w) \to 0 \quad \text{as} \quad z\to\partial\widehat{T_{B}}.
		\]
		\item[(d)]  For every $z\in T_{B}$ there is a complex number $c_z$ such that
		\[
		\frac{1}{|D(z,r)|} \int_{D(z,r)} |f(w)-c_z|^p dV_\alpha(w) \to 0 \quad \text{as} \quad z\to\partial\widehat{T_{B}}.
		\]
	\end{enumerate}
\end{lemma}

And now recall that 
\[
MO_r(f)(z)= \left( \frac{1}{|D(z,r)|} \int_{D(z,r)} |f(w)-\widehat{f}_r(z)|^p dV_\alpha(w) \right)^{1/p},
\]
and
the space $BMO^p_r$ consists of those functions $f\in L^p_\alpha(T_{B})$ such that
\[
\|f\|_{BMO^p_r}= \sup\left\{MO_r(f)(z):z\in T_{B}\right\}.
\]

From Lemma \ref{thm:BlochBMO} and \ref{thm:Bloch_0VMO}, we conclude that the quantities $\|f\|_{\mathcal{B}}$ and $\|f\|_{BMO^p_r}$ are equivalent.
Therefore, {\bfseries Theorem B} is proved.

\begin{remark}
	The equivalences in Theorem A provide a foundational framework for understanding functions in $BMO^p_r$, which directly informs the characterization of the Bloch space in Theorem B. By identifying the Bloch space as the intersection of holomorphic functions and $BMO^p_r$, Theorem B highlights the significance of analytic properties established in Theorem A. Together, these results enhance our understanding of the interrelationships between these function classes and their applications in complex analysis.
\end{remark}

\section{Proofs of the Main Results: Theorems C and D}

To establish that the projection operator $\widetilde{P_\alpha}$ is bounded, we first introduce the integral operator  $T_{a,b,\gamma^{\prime}}$. We will then demonstrate that $T_{a,b,\gamma^{\prime}}$ is bounded on $L_s^p(T_{B})$.

	Given $a,b \in \mathbb{R}$, $\alpha>-1$ and $\gamma^{\prime}\in \mathbb{N}_{0}^{n-1}$, we define the integral operator $T_{a,b,\gamma^{\prime}}$
	by
	\[
	T_{a,b,\gamma^{\prime}} f(z) := \bfrho(z)^{a} \int\limits_{T_{B}} \frac {
		\big|(z^{\prime}-w^{\prime})^{\gamma^{\prime}}\big|  \bfrho (w)^{b} } {|\bfrho(z,w)|^{n+1+a+b+\frac {|\gamma^{\prime}|}{2}}} f(w) dV_\alpha(w),
	\quad z\in T_{B}.
	\]
	
Before we proceed with the proof of the theorem, it is important to highlight Lemma 4.1, which serves as a foundational tool in our analysis. We now state this lemma as a crucial step in our argument.
	\begin{lemma}\label{lem:IntOper}
		Suppose that $1\leq p\leq \infty$ and $s\in\mathbb{R}$. If $-pa<s+1<p(b+1)$, then $T_{a,b,\gamma^{\prime}}$ is bounded on $L_s^p(T_{B})$ .
	\end{lemma}
	
	\begin{proof}
		In the special case when $\gamma^{\prime}=0^{\prime}$,  this has been shown in \cite[Theorem 1]{LD}.
		To prove the general case,  by Lemma \ref{lem:ptwsest}, we have
		\[
		|T_{a,b,\gamma^{\prime}} f(z)| \lesssim T_{a,b,0^{\prime}} (|f|)(z).
		\]
		
	\end{proof}

We now proceed to prove our third result, {\bfseries Theorem C}.
	\begin{thm}\label{thm:main2}
		$\widetilde{P_\alpha}$ is a bounded projection from $L^{\infty }(T_{B})$ to $\widetilde{\mathcal{B}}$.
	\end{thm}
	\begin{proof}
		Let $f\in L^{\infty }(T_{B})$. Obviously, $\widetilde{P_\alpha}f(\mathbf{i})=0$, 
		
		$$\mathcal{L}_j(\widetilde{P_\alpha}f)(z)=\frac{\varGamma \left( n+2+\alpha \right)}{2^{n+2}\pi ^n\varGamma (\alpha+1)}\int_{T_B}\frac{\overline{w}_j-\overline{z}_j}{\bfrho (z,w)^{n+2+\alpha}}f(w)dV_\alpha(w), $$
		for all $j\in \{1, 2, \cdots , n-1\}$.
		
		By a simple calculation, we can get
		\begin{equation*}
			\bfrho (z)^{\frac{1}{2}}\left | \mathcal{L}_j(\widetilde{P_\alpha}f)(z)\right |\leq \frac{\varGamma \left( n+2+\alpha \right)}{2^{n+2}\pi ^n\varGamma (\alpha+1)}\bfrho (z)^{\frac{1}{2}}\int_{T_B} \frac{\left |\overline{w}_j-\overline{z}_j\right |} {\left |\bfrho (z,w)\right |^{n+2+\alpha}}dV_\alpha(w)\cdot \left \| f\right \|_{\infty},
		\end{equation*}
		and
	\begin{equation*}
		\bfrho (z)\left | \mathcal{L}_n(\widetilde{P_\alpha}f)(z)\right |\leq \frac{\varGamma \left( n+2+\alpha \right)}{2^{n+2}\pi ^n\varGamma (\alpha+1)}\bfrho (z)\int_{T_B} \frac{1} {\left |\bfrho (z,w)\right |^{n+2+\alpha}}dV_\alpha(w)\cdot
		\left \| f\right \|_{\infty}.
	\end{equation*}
	By applying Lemma \ref{lem:IntOper}, we obtain $\left \| \widetilde{P_\alpha }f\right \|_\mathcal{B}\leq C\left \| f\right \|_{\infty }$, which completes the proof.
	\end{proof}

We now proceed to prove the last result, \textbf{Theorem D}. First, we introduce the following proposition, which establishes the boundedness of a crucial linear operator.
\begin{proposition}\label{prop:bddnsofder}
	Suppose $1\leq p < \infty$ , $t>-1$ , for any $\gamma\in \mathbb{N}_{0}^n$,  then the map $f\mapsto \bfrho^{\langle \gamma \rangle} \mathcal{L}^{\gamma} f$
	is a bounded linear operator from $A_t^p(T_{B})$ into $L_t^p(T_{B})$.
	In particular, for any $N\in \mathbb{N}_0$,
	the map $f\mapsto \mathcal{L}_n^N f$ is a bounded linear operator from $A_t^p(T_{B})$ into $A_{t+Np}^p(T_{B})$.
\end{proposition}	

\begin{proof}
	Let $\lambda$ be sufficiently large such that $p(\lambda+1)>t+1$. Then by Lemma \ref{projection}, we have
	\[
	f(z) = c_\lambda \int\limits_{T_{B}} \frac{\bfrho(w)^{\lambda}}{\bfrho(z,w)^{n+1+\lambda}} f(w)dV(w),
	\]
	where $c_\lambda=\frac{\varGamma \left( n+1+\lambda \right)}{2^{n+1}\pi ^n\varGamma (\lambda+1)}$.
	For any $\gamma\in \mathbb{N}_{0}^n$, 
	\begin{equation*}
		\begin{aligned}
			\mathcal{L}^{\gamma} \left\{\frac {1}{\bfrho(z,w)^{n+1+\lambda}} \right\}  &= C(n,\lambda,\gamma)
		\frac {\left(\frac{1}{2}({z}^{\prime}-\overline{{w}^{\prime}})-y'i\right)^{\gamma^{\prime}}}
		{\bfrho(z,w)^{n+1+\lambda+|\gamma|}}\\&=C(n,\lambda,\gamma)
		\frac{\left(\frac{1}{2}({z}^{\prime}-\overline{{z}^{\prime}}+\overline{{z}^{\prime}}-\overline{{w}^{\prime}})-y'i\right)^{\gamma^{\prime}}}{\bfrho(z,w)^{n+1+\lambda+|\gamma|}}\\
		&=C(n,\lambda,\alpha)
		\frac {\left(\overline{{z}^{\prime}}-\overline{{w}^{\prime}}\right)^{\gamma^{\prime}}}
		{\bfrho(z,w)^{n+1+\lambda+|\gamma|}}.
		\end{aligned}
	\end{equation*}
	Here, $C(n,\lambda,\gamma)$ is a constant depending on $n$, $\lambda$ and $\gamma$. Additionally, note that $|\gamma|=\langle \gamma \rangle+|\gamma^{\prime}|/2$.
	Therefore,
	\begin{align*}
		\big|\bfrho^{\langle \gamma \rangle} \mathcal{L}^{\gamma} f(z)\big|
		~\lesssim~& \bfrho(z)^{\langle \gamma \rangle}
		\int\limits_{T_{B}} \frac { \big|(z^{\prime}-w^{\prime})^{\gamma^{\prime}}\big| \bfrho(w)^{\lambda}}
		{|\bfrho(z,w)|^{n+1+\lambda+|\gamma|}} |f(w)| dV(w)\\
		~=~& T_{\langle \gamma \rangle, \lambda, \gamma^{\prime}}(|f|)(z).
	\end{align*}
The proposition now follows directly from Lemma \ref{lem:IntOper}.
\end{proof}

\begin{thm}\label{thm:1}
	Suppose that $1\leq p < \infty$, $\lambda>-1$ and $\alpha\in \mathbb{R}$ satisfy
	\[
	\begin{cases}
		\alpha > \frac {\lambda+1}{p}-1,& 1<p<\infty,\\
		\alpha \geq \lambda,& p=1.
	\end{cases}
	\]
	If $f\in A_\lambda^p(T_{B})$ then
	\begin{equation*}\label{eqn:intrepn1}
		f=  \frac{(2i)^N\Gamma(1+\alpha)}{\Gamma(1+\alpha+N)}\mathcal{T} _\alpha (\bfrho^N \mathcal{L}_n^N f)
	\end{equation*}
	for any $N\in\mathbb{N}_0$.
\end{thm}

To prove this theorem, we divide the proof into two separate lemmas, each addressing a distinct case. We begin by considering the case where $p=1$ and $\lambda=\alpha$.

\begin{lemma}\label{lem:Prhopar}
	Theorem \ref{thm:1} holds for $p=1$ with $\lambda=\alpha$.
\end{lemma}
\begin{proof}
	Suppose  that $f\in A_{\alpha}^1(T_{B})$. Let $\gamma>\alpha$, by Lemma \ref{projection}, we have
	\[
	f(z)=c_{\gamma} \int\limits_{T_{B}} \frac {\bfrho(w)^{\gamma}}{\bfrho(z,w)^{n+1+\gamma}} f(w) dV(w)
	\]
	for all $z\in T_{B}$. Then it is easy to shows that
	\[
	\mathcal{L}_n^N f(z) = (n+1+\gamma)_{N} (-i/2)^N c_{\gamma}  \int\limits_{T_{B}} \frac {\bfrho(w)^{\gamma}}{\bfrho(z,w)^{n+1+\gamma+N}} f(w) dV(w),
	\]
	where $(n+1+\gamma)_{N}$ is the Pochhammer symbol denoted by $(a)_{k}=a(a+1)\cdots(a+k-1)$.
	Thus,
	\begin{align*}
	\mathcal{T} _\alpha (\bfrho^N \mathcal{L}_n^N f)(z)
		&= c_{\alpha} \int\limits_{T_{B}} \frac{\bfrho(w)^{\alpha+N} \mathcal{L}_n^N f(w)}{\bfrho(z,w)^{n+1+\alpha}} dV(w)\\
		&=  c_{\alpha} (n+1+\gamma)_{N} (-i/2)^N c_{\gamma} \\
		& \quad\times  \int\limits_{T_{B}} \frac{\bfrho(w)^{\alpha+N}}{\bfrho(z,w)^{n+1+\alpha}}
		\bigg(\int\limits_{T_{B}} \frac {\bfrho(u)^{\gamma}}{\bfrho(w,u)^{n+1+\gamma+N}} f(u) dV(u) \bigg) dV(w).
	\end{align*}
	By Fubini's theorem  and Lemma \ref{eqn:keylem2},  for any $z\in T_{B}$, the above double integral equals to
	\begin{align*}
		\int\limits_{T_{B}} & \bfrho(u)^{\gamma} f(u) \bigg(
		\int\limits_{T_{B}} \frac{\bfrho(w)^{\alpha+N}}{\bfrho(z,w)^{n+1+\alpha} \bfrho(w,u)^{n+1+\gamma+N}} dV(w)
		\bigg) dV(u)\\
		&=  C_1(n,n+1+\alpha,n+1+\gamma+N,\lambda+N)
		\int\limits_{T_{B}} \frac{\bfrho(u)^{\gamma} f(u)}{\bfrho(z,u)^{n+1+\gamma}} dV(u)\\
		&= C_1(n,n+1+\alpha,n+1+\gamma+N,\alpha+N) c_{\gamma}^{-1} f(z).
	\end{align*}
 Consequently,
	\[
\mathcal{T} _\alpha (\bfrho^N \mathcal{L}_n^N f)(z)=\frac{(-i/2)^N\Gamma(1+\alpha+N)}{\Gamma(1+\alpha)} f(z).
	\]
	
We now justify the use of Fubini's theorem. For any fixed $z\in T_{B}$, by Lemma \ref{eqn:keylem2} (note that $\gamma>\alpha$) and \eqref{essiential}, it follows that
	\begin{align*}
		&\int_{T_{B}} \bigg( \int_{T_{B}} \frac{\bfrho(w)^{\alpha+N}}{|\bfrho(z,w)|^{n+1+\alpha} |\bfrho(w,u)|^{n+1+\gamma+N}} dV(w) \bigg) \bfrho(u)^{\gamma} |f(u)| dV(u)\\
		&\leq (\bfrho(z)/2)^{-n-1-\alpha} \int_{T_{B}} \bigg( \int_{T_{B}} \frac{\bfrho(w)^{\alpha+N}}{|\bfrho(w,u)|^{n+1+\gamma+N}} dV(w) \bigg) \bfrho(u)^{\gamma} |f(u)| dV(u)\\
		&=  (\bfrho(z)/2)^{-n-1-\alpha} C_1(n,n+1+\gamma+N,\alpha+N)
		\int_{T_{B}} |f(u)| \bfrho(u)^{\alpha}dV(u) <\infty .
	\end{align*}
	The proof of the lemma is completed.
\end{proof}

	By applying  Lemma \ref{lem:Prhopar}, we can now address the remaining part of the proof.
	
	\begin{lemma}\label{lem:Prhopar}
		Theorem \ref{thm:1} holds for $1<p<\infty$ when $p(\alpha+1)>\lambda+1$, and for $p=1$ when  $\alpha>\lambda$.
	\end{lemma}

\begin{proof}
	Put
	\[
	g:=f - \frac{(2i)^N\Gamma(1+\alpha)}{\Gamma(1+\alpha+N)}\mathcal{T} _\alpha  (\bfrho^N \mathcal{L}_n^N f).
	\]
	Note from Proposition \ref{prop:bddnsofder} that $\mathcal{L}_n^N f\in A_{\lambda+Np}^p(T_{B})$. It is not difficult to see that $\mathcal{T} _\alpha  (\bfrho^N \mathcal{L}_n^N f)$ belongs to $A_\lambda^p(T_{B})$ and so does $g$.
	A simple calculation shows that
\begin{equation}\label{key1}
	\begin{aligned}
	\mathcal{L}_n^N g(z) &= \mathcal{L}_n^N f(z) - c_{\alpha+N}
	\int_{T_{B}} \frac {\bfrho(w)^{\alpha+N}} {\bfrho(z,w)^{n+1+\alpha+N}} \mathcal{L}_n^N f (w) dV(w)\\
	&= \mathcal{L}_n^N f(z) - \mathcal{T} _{\alpha+N}(\mathcal{L}_n^N f )(z)=0 .
    \end{aligned}
\end{equation}
	By \cite{DK}, since $p(\alpha+N+1)>\lambda+Np+1$ for any case.
	
	Suppose $\mathcal{L}_n^N g\equiv 0$. Then $g$ has the form
	\[
	g(z)=g_{N-1}(z^{\prime}) z_n^{N-1}+g_{N-2}(z^{\prime}) z_n^{N-2}+\cdots+ g_0(z^{\prime})
	\]
	with $g_{N-1}, \cdots, g_0$ holomorphic functions of $z^{\prime}$.
	It follows that
	\[
	\mathcal{L}_n^{N-1} g(z) = (N-1)! g_{N-1}(z^{\prime}).
	\]
	All that remains is to show  that $g_{N-1}(z^{\prime})\equiv 0$. In view of Proposition \ref{prop:bddnsofder}, we know that $\bfrho^{N-1} \mathcal{L}_n^{N-1} g\in L_\lambda^p(T_{B})$. Thus,
	\begin{align*}
		\infty &> \int_{T_{B}}  \bfrho(z)^{(N-1)p+\lambda} \left| g_{N-1}(z^{\prime})\right|^p dV(z)\\
		&=
		\int_{\mathbb{R}^n}\int_{\mathbb{R}^{n-1}}\int_{y_n>y'^2}(y_n-y'^2)^{(N-1)p+\lambda}dy_n|g_{N-1}(x'+iy')|^pdy'dx .
	\end{align*}
	However, note that the inner integral diverges for any fixed  $y^{\prime}$, implying that $g_{N-1}(z^{\prime})\equiv 0$.
	Consequently, after a finite number of iterations, we conclude that  $g \equiv 0$, thus completing the proof.
\end{proof}
By combining the two previously discussed lemmas, we have successfully completed the proof of Theorem D.

\end{document}